\documentclass[11pt]{amsart}
\usepackage{amsmath,amsfonts,amsthm,amssymb,amscd,comment,xypic,verbatim}

\numberwithin{equation}{section}

\newtheorem{thm}[equation]{Theorem}

\newtheorem{lemma}[equation]{Lemma}
\newtheorem{prop}[equation]{Proposition}

\theoremstyle{remark}
\newtheorem*{remark}{Remark}

\newtheorem{question}[equation]{Question}

\newcommand{\abs}[1]{\left\lvert#1\right\rvert}

\renewcommand{\bar}[1]{#1\llap{$\overline{\phantom{\rm#1}}$}}

\newcommand{\RR}{\ensuremath{\mathbb{R}}}
\newcommand{\CC}{\ensuremath{\mathbb{C}}}
\newcommand{\Z}{\ensuremath{\mathbb{Z}}}
\newcommand{\QQ}{\mathbb{Q}}

\newcommand{\QB}{\bar{\mathbb{Q}}}
\newcommand{\Line}{\mathbb{P}^1}

\DeclareMathOperator{\gal}{Gal}
\DeclareMathOperator{\res}{Res}
\DeclareMathOperator{\br}{Br}

\title{Separated Belyi Maps}

\author{Zachary Scherr}
\address{
  Department of Mathematics,
  University of Pennsylvania,
  209 South 33rd Street,
  Philadelphia, PA 19104-6395 USA
}
\email{zscherr@math.upenn.edu}
\urladdr{http://www.math.upenn.edu/$\sim$zscherr/}
\author{Michael E. Zieve}
\address{
  Department of Mathematics,
  University of Michigan,
  530 Church Street,
  Ann Arbor, MI 48109-1043 USA
}
\address{
Mathematical Sciences Center, Tsinghua University, Beijing 100084, China}

\email{zieve@umich.edu}
\urladdr{http://www.math.lsa.umich.edu/$\sim$zieve/}

\date{\today}

\thanks{The authors thank the NSF for support under grants EMSW21-RTG:0943832 and DMS-1162181.}

\begin{document}


\begin{abstract}
We construct Belyi maps having specified behavior at finitely many points.  Specifically,
for any curve $C$ defined over $\QB$, and any disjoint finite subsets $S,T\subset C(\QB)$, we construct
a finite morphism $\varphi\colon C\to\Line$ such that $\varphi$ ramifies at each point in $S$, the
branch locus of $\varphi$ is $\{0,1,\infty\}$, and $\varphi(T)\cap\{0,1,\infty\}=\emptyset$.  This refines a result of Mochizuki's.
We also prove an analogous result over fields of positive characteristic, and in addition we analyze how many different Belyi maps $\varphi$
are required to imply the above conclusion for a single $C$ and $S$ and all sets $T\subset C(\QB)\setminus S$ of prescribed cardinality.
\end{abstract}

\maketitle


\section{Introduction}

Let $C$ be a (smooth, projective, geometrically irreducible) algebraic curve defined over $\CC$.  Belyi \cite{Belyi} gave an unexpected necessary and sufficient condition
for $C$ to be isomorphic to a curve defined over $\QB$: namely,  that there should exist a finite morphism $\varphi\colon C\to\Line$ which has
exactly three branch points.  We refer to such a map $\varphi$ as a \textit{Belyi map}, and write $\br(\varphi)$ for its branch locus.
Belyi maps have had important consequences to topics ranging from
Galois theory \cite{Grothendieck} to physics \cite{physics};
see for instance \cite{lift,Elkies,Goldring,diffeq,S,SL,SL2} for various other consequences of Belyi maps.
In some applications, one needs Belyi maps $\varphi$ which satisfy
additional properties.  One prominent example is Mochizuki's 
inter-universal Teichm\"{u}ller theory \cite{abc}, which relies on his earlier results on Belyi maps \cite{Mochizuki}.

We are interested in studying the amount of flexibility there is in the choice of a Belyi map $\varphi$ on a prescribed curve $C$.
One way to measure this is through the preimage $\varphi^{-1}(\br(\varphi))$. We will show in Proposition~\ref{fin} that, for any curve $C$ of genus
at least $2$, only finitely many subsets of $C(\CC)$ of any prescribed cardinality can occur as $\varphi^{-1}(\br(\varphi))$ for a Belyi map $\varphi$ on $C$.
These distinguished finite subsets of $C(\CC)$ are the focus of this paper.

We now state our first main result, which refines a result of Mochizuki's \cite[Thm.~2.5]{Mochizuki}.  Throughout this paper, we view all curves as coming equipped with a fixed
embedding into projective space, so that if a curve $C$ is defined over a field $K$ then we can speak of the coordinatewise action of the
absolute Galois group of $K$ on points of $C(\bar{K})$.

\begin{thm}\label{mochizuki}
Let $K$ be a number field, let $C$ be a curve over $K$, and let $S$ and $T$ be finite subsets of $C(\bar{K})$ such that $S$ is disjoint from the set of
$\gal(\bar{K}/K)$-conjugates of elements of $T$.  Then there exists a finite morphism
$\varphi\colon C\rightarrow\mathbb{P}^{1}$
defined over $K$ such that
\begin{itemize}
\item $\br(\varphi)=\{0,1,\infty\}$
\item $\varphi$ is ramified at every point in $S$
\item $\varphi(T)\cap\{0,1,\infty\}=\emptyset$.
\end{itemize}
\end{thm}

In this result, the set $\{0,1,\infty\}$ could be replaced by any prescribed three-element subset of $K$, since there are linear fractional
transformations over $K$ which map any such subset to $\{0,1,\infty\}$.
In \cite[Thm.~2.5]{Mochizuki}, Mochizuki proved a similar result in which the condition on the ramification of $\varphi$ on $S$ is replaced by the weaker condition that $\varphi(S)=\{0,1,\infty\}$.   Our proof is quite different from Mochizuki's, and uses ideas from \cite{us}, although both our proof
and Mochizuki's follow the general outline of all known proofs of Belyi's theorem.

Positive characteristic analogues of Belyi maps seem to have been considered for the first time by Katz \cite[Lemma 16]{Katz} in the context of
the Langlands correspondence for function fields.  In positive characteristic, every curve admits a finite morphism to $\Line$ having just one branch point.
We prove the following analogue of Theorem~\ref{mochizuki} in this setting.

\begin{thm} \label{charp}
Let $K$ be a perfect field of characteristic $p>0$, let $C$ be a curve over $K$, and let $S,T$ be finite subsets of $C(\bar{K})$ such that
$S$ is disjoint from the set of $\gal(\bar{K}/K)$-conjugates of elements of $T$.
Then there exists a finite morphism
$\varphi\colon C\rightarrow\mathbb{P}^{1}$
defined over $K$ such that
\begin{itemize}
\item $\br(\varphi)=\{\infty\}$
\item $\varphi$ is ramified at every point in $S$
\item $\infty\notin\varphi(T)$.
\end{itemize}
\end{thm}

Now we fix a set $S$ and let $T$ vary over all subsets of $C(\QB)\setminus S$ of prescribed cardinality $n$.  In Theorem~\ref{mochizuki}
we produced a Belyi map corresponding to any such $T$; our next result shows that in fact only $n+1$ Belyi maps are needed to account for all $T$.

\begin{thm} \label{manymaps}
Let $n\ge 1$ be an integer, and let $C$ be a curve defined over\/ $\QB$. If $S\subset C(\QB)$ is a finite set of points, then there exist finite morphisms
\[
\varphi_{1},\ldots,\varphi_{n+1}\colon C\rightarrow\mathbb{P}^{1}
\]
such that
\begin{itemize}
\item $\varphi_{i}(S)\subseteq\{0,1,\infty\}=\br(\varphi_i)$ for $1\le i\le n+1$
\item For any subset $T\subset C(\QB)$ of cardinality $n$ for which $S\cap T=\emptyset$, there exists an $i\in\{0,1,\ldots,n+1\}$ such that $\varphi_{i}(T)\cap\{0,1,\infty\}=\emptyset.$
\end{itemize}
\end{thm}

This refines a result of Mochizuki's \cite[Cor.~3.1]{Mochizuki}, which showed that there were finitely many such Belyi maps via a compactness argument which did not provide
control of the number of maps.   In Theorem~\ref{mymochpgen} we give an analogous result in positive characteristic.
We note that these results have a topological interpretation.  For any curve $C$ defined over $\QB$, we say that a \textit{Belyi open} subset of $C(\QB)$ is
any set of the form $C(\QB)\setminus\varphi^{-1}(\br(\varphi))$ where $\varphi\colon C\to\Line$ is a Belyi map.  Theorem~\ref{mochizuki} implies that,
for any disjoint finite $S,T\subset C(\QB)$, there is a Belyi open set which contains $T$ but is disjoint from $S$.  Theorem~\ref{manymaps} implies that,
for any finite $S\subset C(\QB)$, there are $n+1$ Belyi open sets $U_1,\dots,U_{n+1}$ such that any $n$-element subset $T$ of $C(\QB)\setminus S$
satisfies $T\subset U_i\subset C(\QB)\setminus S$ for some $i$.

Belyi proved his theorem by constructing $\varphi$ as the composition $\varphi_3\circ\varphi_2\circ\varphi_1$, where $\varphi_1\colon C\to\Line$ is any
finite morphism defined over $\bar{\QQ}$, $\varphi_2\colon\Line\to\Line$ satisfies $B:=\br(\varphi_2)\cup\varphi_2(\br(\varphi_1))\subset\QQ$,
and $\varphi_3\colon\Line\to\Line$ satisfies $\br(\varphi_3)\cup\varphi_3(B)=\{0,1,\infty\}$.  His argument shows that there exist Belyi maps $\varphi$
for which $\varphi^{-1}(\br(\varphi))$ contains any prescribed subset of $C(\bar{\QQ})$.  Both Mochizuki's proof and our proof have a similar structure
to Belyi's proof, in that $\varphi$ is constructed as the composition of three maps.  The difference is that we require some points to stay away
from the branch locus, as in the following diagram:
\[
\xymatrix{C\ar[d]^{\varphi_{1}} & T\ar[d] & S\ar[d]\\
\mathbb{P}^{1}\ar[d]^{\varphi_{2}}&\alpha=\infty\ar[d]\ar @{} [r] |-{\not\in}&A=\varphi_{1}(S)\cup \br(\varphi_1)\ar[d]\\
\mathbb{P}^{1}\ar[d]^{\varphi_{3}}&\beta=\varphi_{2}(\infty)\ar[d]\ar @{} [r] |-{\not\in}&B=\varphi_{2}(A)\cup \br(\varphi_{2})\ar[d]\\
\mathbb{P}^{1}&\varphi_{3}(\beta)\ar @{} [r] |-{\not\in}&\{0,1,\infty\}
}
\]
We will construct the map $\varphi_1$ in the next section, treating the case of positive characteristic at the same time as the case of
characteristic zero.  The maps $\varphi_2$ and $\varphi_3$ are constructed for fields of characteristic zero in the section after that,
and then we construct them over fields of positive characteristic.  Then in the last section we treat collections of Belyi maps.


\section{Reduction to $\mathbb{P}^{1}$}

Our goal is to show that every curve $C$ admits a Belyi map $C\rightarrow\mathbb{P}^1$ satisfying additional constraints.  In this section we show that it suffices to do this when $C=\mathbb{P}^1$.

\begin{prop}\label{riemannrochlemma}
Let $C$ be a curve defined over a perfect field $K$.  If $S$ and $T$ are disjoint finite $\gal(\bar{K}/K)$-stable subsets of $C(\bar{K})$ then there exists a finite morphism
\[
\varphi\colon C\rightarrow\mathbb{P}^{1}
\]
defined over $K$ such that $\varphi(T)\subseteq\{\infty\}$ and $\infty\not\in\varphi(S)\cup\br(\varphi)$.
\end{prop}

Before proving this we recall some terminology.  A divisor on $C$ (over $\bar{K}$) is said to be \textit{defined over $K$} if it is fixed by
$\gal(\bar{K}/K)$.  For any divisor $D$ on $C$ which is defined over $K$, the associated Riemann--Roch space is the $K$-vector space
\[
\mathcal{L}(D):=\{f\in K(C)\colon (f)\ge -D\}\cup\{0\},
\]
and the dimension of this vector space is denoted $\ell(D)$.

\begin{proof}
It suffices to prove the result when $T$ is nonempty, since the result for empty $T$ follows from the result for nonempty $T$.
Thus, we will assume that $T$ is nonempty.
Let $O_1,O_2,\ldots,O_n$ be the $\gal(\bar{K}/K)$-orbits of points in $T$, and for each $i$ let $D_i$ be the divisor
\[
D_i=\sum_{P\in O_i}P.
\]
Note that $D_i$ is defined over $K$.  Let $g$ be the genus of $C$, and let $R$ be a finite $\gal(\bar{K}/K)$-stable subset of $C(\bar{K})$ such that $\abs{R}\ge 2g-1$ and $R\cap(S\cup T)=\emptyset$.
Let $D$ be the divisor
\[
D=\sum_{P\in R} P,
\]
so that $D$ is defined over $K$ and $\deg(D)=\abs{R}\ge2g-1$.
By the following lemma, there is an element $f\in \mathcal{L}(D+\sum_{i=1}^n D_i)$ which is not in $\mathcal{L}(D+\sum_{i\ne j} D_i)$ for any $j$.
Then $f\in K(C)$ has simple poles at all the points of $T$, at most simple poles at the points in $R$, and no other poles.  Because $R$ is disjoint from $S$, we see that $f$ extends to a morphism $\varphi$ satisfying the requirements of the proposition.
\end{proof}

\begin{lemma}\label{rrlem1}
Let $C$ be a curve of genus $g$ defined over a perfect field $K$.  Let $T$ be a finite, non-empty $\gal(\bar{K}/K)$-stable set of points of $C(\bar{K})$.  Let $O_1,\ldots, O_n$ be the $\gal(\bar{K}/K)$ orbits in $T$, and let $D_i$ be the divisor $D_i=\sum_{P\in O_i} P$.  If $D$ is a divisor defined over $K$ of degree at least $2g-1$, then
\[
\mathcal{L}(D+\sum_{i=1}^n D_i)\supsetneqq\bigcup_{j=1}^n\mathcal{L}(D+\sum_{i\ne j}D_i).
\]
\end{lemma}

\begin{proof}

For each $j\in\{1,2,\ldots,n\}$, we have
\[
\deg(D+\sum_{i=1}^n D_i) > \deg(D+\sum_{i\ne j} D_i)\ge 2g-1,
\]
so the Riemann-Roch theorem implies that
\[
\ell(D+\sum_{i=1}^n D_i) = \deg(D)+\abs{T}+1-g >\ell(D+\sum_{i\ne j} D_i).
\]
Since a vector space over an infinite field cannot be written as the union of finitely many proper subspaces, it follows that if $K$ is infinite then 
\[
\mathcal{L}(D+\sum_{i=1}^n D_i)\supsetneqq\bigcup_{j=1}^n\mathcal{L}(D+\sum_{i\ne j}D_i).
\]

Henceforth assume that $K$ is finite, and put
\begin{eqnarray*}
m_i&=&\deg(D_i)=\abs{O_i}\\
m&=&\abs{T}\\
q&=&\abs{K}\\
r&=&\deg(D)+1-g.
\end{eqnarray*}
The Riemann--Roch theorem implies that $\ell(D+\sum_{i=1}^n D_i)=r+m$, and thus
\[
\Big\lvert\mathcal{L}(D+\sum_{i=1}^n D_i)\Big\rvert=q^{r+m}.
\]
For distinct $j_1,j_2,\ldots,j_k\in \{1,2,\ldots,n\}$, we have
\[
\bigcap_{t=1}^k\mathcal{L}(D+\sum_{i\ne j_t} D_i)=\mathcal{L}(D+\sum_{i\ne j_1,\ldots,j_k} D_i).
\]

\noindent Riemann--Roch implies that
\begin{eqnarray*}
\ell(D+\sum_{i\ne j_1,\ldots,j_k} D_i)&=&r+\sum_{i\ne j_1,\ldots,j_t} m_i\\
&=&r+m-m_{j_1}-\ldots-m_{j_k},
\end{eqnarray*}
so that
\[
\abs{\mathcal{L}(D+\sum_{i\ne j_1,\ldots,j_k} D_i)} = q^{r+m-m_{j_1}-\ldots-m_{j_k}}.
\]

\noindent It follows by inclusion-exclusion that
\[
\Big\lvert\bigcup_{j=1}^n\mathcal{L}(D+\sum_{i\ne j} D_i)\Big\rvert
=\sum_{k=1}^n (-1)^{k+1}\sum_{1\le j_1<\ldots<j_k\le n}q^{r+m-\sum_{t=1}^k m_{j_t}}.
\]
This cardinality equals
\[
q^{r+m}\left(1-\prod_{i=1}^n\left(1-\frac{1}{q^{m_{j_i}}}\right)\right),
\]
which is strictly smaller than $q^{r+m}$.  Thus our union of subspaces is a proper subset of $\mathcal{L}(D+\sum_{i=1}^n D_i)$.
\end{proof}

\section{Characteristic 0}

In this section we prove that for any finite subset $A\subset\mathbb{P}^1(\QB)$, and any element $\alpha\in\mathbb{P}^1(\QB)\setminus A$, there is a Belyi map $\varphi\colon\mathbb{P}^1\rightarrow\mathbb{P}^1$ with $\varphi(\alpha)\notin\br(\varphi)=\{0,1,\infty\}$ such that all points in $A$
ramify under $\varphi$.  We do this in two steps.  The first step produces a map $\varphi_2$ which allows us to reduce to the case that $A$ and $\alpha$ are in $\Line(\QQ)$.  The second step produces a map $\varphi_3$ for which $\varphi=\varphi_3\circ\varphi_2$ has the required properties.

To proceed, fix an embedding $\QB\hookrightarrow\CC$, and let $\abs{.}\colon\QB\rightarrow\RR$ be the induced absolute value.  The following lemma 
enables us to control the absolute values of $\alpha$ and of the elements of $A$.

\begin{lemma}\label{sizechange}
Let $A$ be a finite subset of\/ $\mathbb{P}^{1}(\QB)$, and suppose $\alpha\in\mathbb{P}^{1}(\QQ)\setminus A$.  Then for any real number $c>1$, there exists a fractional linear transformation $\psi\in\QQ(x)$ such that $\psi(\alpha)\in\mathbb{Q}$ with $\displaystyle\max_{\beta\in A}\{\abs{\psi(\beta)}\}<1$ and $\abs{\psi(\alpha)}>c$.
\end{lemma} 
\begin{proof}
By applying an appropriate fractional linear transformation with rational coefficients, we may assume without loss of generality that $\alpha=0$ and $\infty\notin A$. Because $\alpha=0$ is not in $A$, we can choose an $r\in\QQ$ with $\displaystyle 0 < r < \min_{\beta\in A}\frac{\abs{\beta}}{c+1}$. Rewriting this inequality, we have that for any $\beta\in A$
\[
rc < \abs{\beta}-r.
\]
Thus the triangle inequality yields
\[
rc < \abs{\beta} - r \le \abs{\beta-r}.
\]
Choose  a rational number $s$ such that $\displaystyle rc < s <\min_{\beta\in A} \abs{\beta-r}$. For the fractional linear transformation
\[
\psi(x)=\frac{s}{x-r}
\]
we see that
\[
\abs{\psi(0)}=\frac{s}{r}>c
\]
and that for $\beta$ in $A$ we have
\[
\abs{\psi(\beta)} = \frac{s}{\abs{\beta-r}}<1.
\]
Since $s$ and $r$ are rational, $\psi$ satisfies the conditions of the lemma.
\end{proof}

Our next result enables us to reduce from the case of a finite subset of $\mathbb{P}^1(\QB)$ to a finite subset of $\mathbb{P}^1(\QQ)$.

\begin{prop}\label{reductiontoQ}
Let $A$ be a finite subset of\/ $\mathbb{P}^{1}(\QB)$, and let $\alpha\in\mathbb{P}^{1}(\QQ)\setminus A$.  There exists a rational function $f\in\QQ(x)$ satisfying
\[
f(A), \br(f), f(\alpha)\subseteq\mathbb{P}^{1}(\QQ)\quad\text{and}\quad
f(\alpha)\not\in f(A)\cup \br(f).
\]
\end{prop}
\begin{proof}
If $A=\emptyset$, then we lose nothing by replacing $A$ with $\{\beta\}$ for some $\beta\in\QQ$ different from $\alpha$.  Thus for the remainder of the proof we assume that $A$ is nonempty. Let $\bar{A}$ be the set of all Galois conjugates of elements of $A$.  Since $\alpha\in\mathbb{P}^{1}(\QQ)\setminus A$ we know that $\alpha\not\in\bar{A}$ as well.  Applying a fractional linear transformation from Lemma~\ref{sizechange},
we may assume without loss of generality that $\abs{\beta}<1$ for all $\beta\in\bar{A}$, and that $c<\abs{\alpha}<\infty$ for some positive constant $c$ to be chosen later. Note that the fractional linear transformation has rational coefficients so the image of $\bar{A}$ is still Galois stable.
Writing $n:=\#\bar{A}$, we will define polynomials $f_0,f_1,\dots,f_{n-1}\in\QQ[x]$ such that
\begin{enumerate}
\item $\deg(f_i)=n-i$ for $0\le i\le n-1$
\item $f_i(\br(f_{i-1}))=\{0\}$ for $1\le i\le n-1$
\item $f_0(\bar{A})=\{0\}$
\item the leading coefficient of $f_i$ is an integer
\item all coefficients of $f_i$ have absolute value bounded by a function of $n$ (independent of the choice of $c$) for $0\le i\le n-1$.
\end{enumerate}
Supposing for the moment that such $f_i$ have been constructed, we now show that if $c$ is sufficiently large compared to $n$ then
$f:=f_{n-1}\circ f_{n-2}\circ\dots\circ f_1\circ f_0$ satisfies the conclusion of the Proposition.  For $\beta\in \bar{A}$
we have $f_0(\beta)=0$ and thus $f(\beta)=f_{n-1}\circ \dots\circ f_1(0)$ is a rational number whose absolute value is bounded by a
function of $n$.  Next, any finite element $\delta$ of $\br(f)$ has the form $\delta=f_{n-1}\circ f_{n-2}\circ\dots\circ f_i(\gamma)$ where $0\le i\le n-2$
and $f_i'(\gamma)=0$. Then $f_{i+1}(f_i(\gamma))=0$, so $\delta=f_{n-1}\circ f_{n-2}\circ\dots\circ f_{i+2}(0)$ is a rational number whose absolute
value is bounded by a function of $n$.
Finally, since $\alpha\in\QQ$ and $c<\abs{\alpha}<\infty$,
we see that $f(\alpha)\in\QQ$ and that $\abs{f(\alpha)}$ is larger than any prescribed function of $n$ whenever $c$ is sufficiently large;
since $f(A)$ and $\br(f)\setminus\{\infty\}$ are bounded by a function of $n$, the statement of the result follows.

To finish the proof, we must construct polynomials $f_i\in\QQ[x]$ satisfying the stated properties.  Define
\[
f_{0}(x):=\prod_{\beta\in\overline{A}}(x-\beta).
\]
Since $\bar{A}$ is Galois stable, we see that $f_{0}$ is a monic polynomial in $\QQ[x]$ of degree $n$.  Moreover, $f_0(\bar{A})=\{0\}$ and
all coefficients of $f_0$ have absolute value bounded by a function of $n$.
Inductively, suppose for any $1\le i\le n-1$ that $f_{i-1}$ has been defined and satisfies the stated properties.  Then put
\[
f_{i}(x) := \res_{y}(f_{i-1}'(y),f_{i-1}(y)-x),
\]
where $\res_{y}()$ is the resultant with respect to the variable $y$.  Since the resultant is the determinant of the Sylvester matrix,
we see that $f_i$ has rational coefficients and that all its coefficients are bounded by a function of $n$.
Let $ax^d$ be the leading term of $f_{i-1}$, where we know that $d=n-i+1$.  Let $r_1,r_2,\dots,r_{d-1}$ be the roots of $f'_{i-1}$, counted with
multiplicity.  Then we have
\[
f_i(x)=(ad)^d\prod_{j=1}^{d-1}(f_{i-1}(r_j)-x).
\]
It follows that the leading term of $f_i$ is $(-x)^{d-1}(ad)^d$, and that the roots of $f_i$ are precisely the finite branch points of $f_{i-1}$.
This concludes the proof.
\end{proof}

We now construct the desired Belyi map in case the curve is $\mathbb{P}^1$ and all distinguished points are in $\mathbb{P}^1(\QQ)$.

\begin{prop}\label{finishQ}
Let $B$ be a finite subset of\/ $\mathbb{P}^{1}(\QQ)$, and suppose $\beta\in\mathbb{P}^{1}(\QQ)\setminus B$.  There exists a rational function $f\in\QQ(x)$ satisfying:
\begin{itemize}
\item $f$ is ramified at every point in $B$
\item $\br(f)=\{0,1,\infty\}$
\item $f(\beta)\notin\{0,1,\infty\}$.
\end{itemize}
\end{prop}

\begin{proof}
By applying a fractional linear transformation of the form $\frac{1}{x-a}$ if necessary (with $a\in\QQ$), we may assume that $B\subseteq\QQ$ and $\beta\in\QQ$.  We may also apply a fractional linear transformation of the form $ax$ with $a\in\Z\setminus\{0\}$ to ensure that $B\subseteq\Z$ and $\beta\in\Z$.
By adjoining to $B$ a sufficiently large finite subset of $\Z\setminus (B\cup\{\beta\})$, we may assume that $\#B\ge 2$.
Let $p$ be a prime which does not divide $\gamma-\beta$ for any $\gamma\in B$, and put $\delta:=\beta+p$.  Let $B':=B\cup\{\delta\}$, and write the elements of $B'$ as
\[
B'=\{b_{1},b_{2},\ldots,b_{m}\}
\]
with $b_{m}=\delta$.

We now construct the desired rational function $f(x)$.  Let
\[
c=\prod_{\substack{1\le i,j\le m \\ i\ne j}} (b_{i}-b_{j}),
\]
and note that $c\ne 0$.
Partial fraction decomposition ensures that there are unique $n_1,\dots,n_m\in\QQ$ such that
\[
\sum_{i=1}^m\frac{n_i}{x-b_i}=\frac{c}{\prod_{i=1}^m (x-b_i)},
\]
namely
\[
n_i=\frac{c}{\displaystyle{\prod_{\substack{1\le j\le m \\ j\ne i}} (b_j-b_i)}}.
\]
Our choice of $c$ ensures that the $n_{i}$ are nonzero integers.  Finally, we define
\[
f(x)=\prod_{i=1}^m (x-b_i)^{2n_i}.
\]

It remains to check that $f(x)$ has the required properties.  Since each $n_i$ is nonzero, the elements of $B'$ are critical points of $f$.
Every finite critical point of $f$ must be either a zero or a pole of
\[
\frac{f'(x)}{f(x)}=\sum_{i=1}^m\frac{2n_i}{x-b_i}.
\]
Our construction of $f$ ensures that the right side equals
\[
\frac{2c}{\prod_{i=1}^m (x-b_i)},
\]
so $f$ cannot have any critical points outside of $B'\cup\{\infty\}$.  Since $f(B')\subseteq\{0,\infty\}$ and $f(\infty)\in\{0,1,\infty\}$, it follows that $\br(f)\subseteq\{0,1,\infty\}$.  Moreover, since $f$ ramifies at each element of $B'$ and we know that $\#B'\ge 3$, Riemann--Hurwitz implies that $f$
must have at least three branch points, whence $\br(f)=\{0,1,\infty\}$.  Finally, we show that $f(\beta)\notin\{0,1,\infty\}$.
Since $\beta\not\in B'$, we know {\it a priori} that $f(\beta)\not\in\{0,\infty\}$.  Our choice of the prime $p$ ensures that $p\nmid(\beta-b_i)$ for $i<m$,
yet $p\mid(\beta-b_m)$. Thus the $p$-adic valuation of
\[
f(\beta)=\prod_{i=1}^m(\beta-b_i)^{2n_i}
\]
is nonzero, so that $f(\beta)\ne 1$.  
\end{proof}

With all the ingredients in place, we can now prove Theorem~\ref{mochizuki}.

\begin{proof}[Proof of Theorem~\ref{mochizuki}]
By adjoining to $S$ the set of $\gal(\bar{K}/K)$-conjugates of elements of $S$, we may assume that $S$ is preserved by
$\gal(\bar{K}/K)$.  Likewise we may assume that $T$ is preserved by $\gal(\bar{K}/K)$.
By adjoining to $T$ a finite $\gal(\QB/K)$-stable set of points in $C(\QB)\setminus S$ if necessary, we may assume that $T$ is nonempty.  
Let $\varphi_{1}\colon C\rightarrow\mathbb{P}^{1}$ be the map produced by Proposition~\ref{riemannrochlemma} for this choice of $S$ and $T$.
Then $\varphi_1$ is defined over $K$, the set $A:=\varphi_1(S)\cup\br(\varphi_1)$ does not contain $\alpha:=\infty$, and $\varphi_1(T)=\{\infty\}$.
Let $\varphi_{2}\colon\mathbb{P}^{1}\rightarrow\mathbb{P}^{1}$ be the map produced by
Proposition~\ref{reductiontoQ} for $A$ and $\alpha$.
Then $\varphi_2$ is defined over $\QQ$, both $\beta:=\varphi_2(\alpha)$ and $B:=\varphi_2(A)\cup\br(\varphi_2)$ are contained in $\mathbb{P}^1(\QQ)$,
and $\beta\notin B$.
Lastly, let $\varphi_{3}\colon \mathbb{P}^{1}\rightarrow\mathbb{P}^{1}$ be the map produced by Proposition~\ref{finishQ} for $B$ and $\beta$.
Then $\varphi_3$ is defined over $\QQ$, every point in $B$ ramifies under $\varphi_3$, and $\varphi_3(\beta)\notin\{0,1,\infty\}=\br(\varphi_3)$.
Pictorially, we have
\[
\xymatrix{C\ar[d]^{\varphi_{1}} & T\ar[d] & S\ar[d]\\
\mathbb{P}^{1}\ar[d]^{\varphi_{2}}&\alpha=\infty\ar[d]\ar @{} [r] |-{\not\in}&A=\varphi_{1}(S)\cup \br(\varphi_1)\ar[d]\\
\mathbb{P}^{1}\ar[d]^{\varphi_{3}}&\beta=\varphi_{2}(\infty)\ar[d]\ar @{} [r] |-{\not\in}&B=\varphi_{2}(A)\cup \br(\varphi_{2})\ar[d]\\
\mathbb{P}^{1}&\varphi_{3}(\beta)\ar @{} [r] |-{\not\in}&\{0,1,\infty\}
}
\]
Thus $\varphi:=\varphi_{3}\circ\varphi_{2}\circ\varphi_{1}$ is a finite morphism $C\to\Line$ defined over $K$, and
$\varphi(T)=\varphi_3(\beta)\notin\{0,1,\infty\}$.  Since $\varphi_2(\varphi_1(S))\subseteq B$ and every point of $B$ ramifies
under $\varphi_3$, it follows that every point of $S$ ramifies under $\varphi$.  Finally, $\br(\varphi)$ is the union of the three sets
$\br(\varphi_3)$, $\varphi_3(\br(\varphi_2))$, and $\varphi_3(\varphi_2(\br(\varphi_1)))$, and hence equals $\{0,1,\infty\}$.
\end{proof}

We conclude this section by showing that, if $C$ is a complex curve of genus at least $2$, then only finitely many subsets of $C(\CC)$ of
any prescribed cardinality can occur as $\varphi^{-1}(\br(\varphi))$ where $\varphi\colon C\to\Line$ is a Belyi map.  In fact we will show at
the same time that only finitely many curves $C$ of any prescribed genus admit a Belyi map $\varphi\colon C\to\Line$ for which
$\varphi^{-1}(\br(\varphi))$ has prescribed cardinality.

\begin{prop} \label{fin}
Fix integers $n,g\ge 1$.  There are only finitely many isomorphism classes of complex curves $C$ of genus $g$ for which there exists a Belyi map
$\varphi\colon C\to\Line$ such that $\#\varphi^{-1}(\br(\varphi))=n$.  Moreover, if $g\ge 2$ and $C$ is a curve of genus $g$, then there
are only
finitely many $n$-element subsets of $C(\CC)$ which occur as $\varphi^{-1}(\br(\varphi))$ for a Belyi map $\varphi\colon C\to\Line$.
\end{prop}

\begin{proof}
Let $\varphi\colon C\to\Line$ be a degree-$d$ Belyi map on a genus-$g$ curve $C$, and let $B:=\varphi^{-1}(\br(\varphi))$.
For any $P\in B$, write $e(P)$ for the ramification index of $P$ under $\varphi$.  The Riemann--Hurwitz formula implies that
\[
2g-2 = -2d + \sum_{P\in B} (e(P)-1) = -2d + 3d-\#B,
\]
so that $d=2g-2+\#B$.  Thus, $d$ is determined by $g$ and $n:=\#B$.

For any Belyi map $\varphi\colon C\to\Line$, there is a fractional linear transformation
$\mu\in\CC(x)$ for which $\mu(\br(\varphi))=\{0,1,\infty\}$.  Thus
$\bar\varphi:=\mu\circ\varphi$ is a Belyi map with branch locus $\{0,1,\infty\}$, and
$\varphi^{-1}(\br(\varphi))=\bar\varphi^{-1}(\{0,1,\infty\})$.  So there is no loss in restricting to Belyi maps with branch locus $\{0,1,\infty\}$.
A classical result (see for instance \cite[Prop.~3.1]{Koeck}) implies that, for any fixed $d$, there are only
finitely many isomorphism classes of pairs $(C,\varphi)$ where $C$ is a complex curve and $\varphi\colon C\to\Line$ is a degree-$d$ Belyi map with
branch locus $\{0,1,\infty\}$.  Thus, for fixed $n$ and $g$, there are only finitely many isomorphism classes of corresponding curves $C$.
Since any curve of genus at least $2$ has only finitely many automorphisms, it follows that for fixed $n,g$ with $g\ge 2$ and a fixed genus-$g$
curve $C$ there are only finitely many $n$-element sets of the form $\varphi^{-1}(\{0,1,\infty\})$ where $\varphi\colon C\to\Line$
is a Belyi map with branch locus $\{0,1,\infty\}$.
\end{proof}


\section{Characteristic $p$}

In this section we prove Theorem~\ref{charp}.  The key tool is the following result:

\begin{prop}\label{charpprop}
Let $K$ be a perfect field of characteristic $p>0$, and let $B\subseteq\mathbb{P}^{1}(\bar{K})$ be a finite set such that $0\not\in B$. Then there exists $f\in K(x)$
such that
\begin{itemize}
\item $f$ is ramified at every point in $B$
\item $\br(f)=\{\infty\}$
\item $f(0)\ne\infty$.
\end{itemize}
\end{prop}

\begin{proof}
Let $\bar{B}\subseteq\mathbb{P}^{1}(\bar{K})$ be the set of $\gal(\bar{K}/K)$-conjugates of elements in $B\setminus\{\infty\}$, and
let $V\subseteq \bar{K}$ be the $\mathbb{F}_{p}$-span of $\bar{B}$.  This is a finite set which is preserved by $\gal(\bar{K}/K)$, and
its minimal polynomial
\[
T(x):=\prod_{\alpha\in V}(x-\alpha)
\]
has the form
\[
T(x) = \sum_{i=0}^{n}a_{i}x^{p^{i}}
\]
where $a_i\in K$ and $a_{0}\ne 0$. Let
\[
S(x) := \frac{T(x)^{p}}{x^{p}} = \sum_{i=0}^{n} a_{i}^{p}x^{p(p^{i}-1)}
\]
and
\[
g(x) := x^{p}+\frac{1}{T(x)+S(x)}.
\]
The rational function $g(x)$ is well-defined since $S(0)=a_{0}^p\ne 0$ while $T(0)=0$.
Plainly $g(\infty)=\infty$.  Since $0\notin B$ by hypothesis, for $\beta$ in $B\setminus\{\infty\}$ we have
\[
T(\beta)+S(\beta) = T(\beta)+\frac{T(\beta)^{p}}{\beta^{p}}=0,
\]
so that $g(\beta)=\infty$.  Thus $g(B)\subseteq\{\infty\}$, and 
furthermore 
\[
g(0)=\frac{1}{T(0)+S(0)}=\frac{1}{a_{0}^p}\ne\infty.
\]
Since $T'(x) = a_{0}$ and $S'(x)=0$, we have
\[
g'(x) = \frac{-a_{0}}{(T(x)+S(x))^{2}}
\]
which has no finite roots since $a_{0}\ne 0$.  Therefore $g(x)$ has no finite critical points which are not poles, and since $g(\infty)=\infty$, we see that $\br(g)\subseteq\{\infty\}$.
The polynomial $x^p+x\in K[x]$ has $\infty$ as its unique critical point and its unique branch point.  Thus $f=g^p+g$ satisfies the requirements of the proposition.
\end{proof}

\begin{proof}[Proof of Theorem~\ref{charp}]
By adjoining to $S$ the set of $\gal(\bar{K}/K)$-conjugates of elements of $S$, we may assume that $S$ is preserved by
$\gal(\bar{K}/K)$.  Likewise we may assume that $T$ is preserved by $\gal(\bar{K}/K)$.
By adjoining to $T$ a finite $\gal(\bar{K}/K)$-stable set of points in $C(\bar{K})\setminus S$ if necessary, we may assume that $T$ is nonempty.  
Let $\varphi_{1}\colon C\rightarrow\mathbb{P}^{1}$ be the map constructed in Proposition~\ref{riemannrochlemma} for this choice of $S$ and $T$.
Then $\varphi_1$ is defined over $K$, the set $A:=\varphi_1(S)\cup\br(\varphi_1)$ does not contain $\alpha:=\infty$, and
$\varphi(T)=\{\infty\}$.  Let $\varphi_2\colon\Line\to\Line$ be the map $x\mapsto 1/x$, and put 
$B:=\varphi_2(A)$, so that $B$ does not contain $\varphi_2(\alpha)=0$.  Let $\varphi_3\colon\Line\to\Line$ be the map constructed in Proposition~\ref{charpprop} for this choice of $B$.  Then $\varphi_3$ is defined over $K$, ramifies at every point in $B$, and
satisfies $\varphi_3(0)\ne\infty$ and $\br(\varphi_3)=\{\infty\}$.  This yields the diagram
\[
\xymatrix{C\ar[d]^{\varphi_{1}} & T\ar[d] & S\ar[d]\\
\mathbb{P}^{1}\ar[d]^{\frac{1}{x}}&\alpha=\infty\ar[d]\ar @{} [r] |-{\not\in}&A=\varphi_{1}(S)\cup \br(\varphi_1)\ar[d]\\
\mathbb{P}^{1}\ar[d]^{\varphi_{3}}&\beta=0\ar[d]\ar @{} [r] |-{\not\in}&B=\{\gamma:1/\gamma\in A\}\ar[d]\\
\mathbb{P}^{1}&\varphi_{3}(\beta)\ar @{} [r] |-{\not\in}&\{\infty\}
}
\]
Then $\varphi:=\varphi_{3}\circ\varphi_2\circ\varphi_{1}$ is a finite morphism $C\to\Line$ defined over $K$,
and $\varphi(T)=\varphi_3(\beta)\ne\infty$.  Since $\varphi_3$ ramifies at every point in $B$, and $B$ contains $\varphi_2(\varphi_1(S))$,
we see that $\varphi$ ramifies at every point in $S$.  Finally, since $\varphi_2$ is unramified, $\br(\varphi)$ is the union of $\br(\varphi_3)$
and $\varphi_3(\varphi_2(\br(\varphi_1)))$, and hence equals $\{\infty\}$.
\end{proof}


\section{Collections of Belyi maps}

Finally, we consider collections of Belyi maps, and prove Theorem~\ref{manymaps}.  

\begin{proof}[Proof of Theorem~\ref{manymaps}]
Let $T_0=\emptyset$.  For each $i=1,2,\dots,n+1$, if $T_{i-1}$ is a finite subset of $C(\QB)\setminus S$ then we define
a morphism $\varphi\colon C\to\Line$ and a finite subset $T_i$ of $C(\QB)\setminus S$ as follows.
Let
$\varphi_i\colon C\rightarrow\mathbb{P}^1$ be the morphism produced by Theorem~\ref{mochizuki} for the sets $S$ and $T:=T_{i-1}$.
Then
\[
\varphi_i(S)\subseteq\{0,1,\infty\}=\br(\varphi_i)
\]
and
\[
\varphi_i(T_{i-1})\cap\{0,1,\infty\}=\emptyset.
\]
Then
\[
T_i:=T_{i-1}\cup \varphi_{i}^{-1}(\{0,1,\infty\})\setminus S
\]
is a finite subset of $C(\QB)\setminus S$.
This procedure yields finite morphisms $\varphi_1,\ldots,\varphi_{n+1}\colon C\rightarrow \mathbb{P}^1$.
For $i=1,2,\dots,n+1$, note that
\[
\varphi_{i}^{-1}(\{0,1,\infty\}) = S\cup\left(T_i\setminus T_{i-1}\right).
\]
Since $T_1\subseteq T_2\subseteq\ldots\subseteq T_{n+1}$, the sets
\[
T_1\setminus T_0, T_2\setminus T_1,\ldots, T_{n+1}\setminus T_n
\]
are pairwise disjoint.  Thus any $n$-element subset $T$ of $C(\QB)$ must be disjoint from at least one set $T_i\setminus T_{i-1}$,
so if $T\cap S=\emptyset$ then $\varphi_i(T)\cap\{0,1,\infty\}=\emptyset$.
\end{proof}

The following analogue of Theorem~\ref{manymaps} can be shown by a similar argument.

\begin{thm}\label{mymochpgen}
Let $n$ be a positive integer, and let $C$ be a curve defined over a perfect field $K$ of characteristic $p>0$. 
For any finite  $S\subset C(K)$, there exist finite morphisms
\[
\varphi_{1},\ldots,\varphi_{n+1}\colon C\rightarrow \mathbb{P}^{1}
\]
defined over $K$ such that
\begin{itemize}
\item $\varphi_{i}(S)\subseteq\{\infty\}=\br(\varphi_i)$ for $1\le i\le n+1$
\item For any $n$-element subset $T\subset C(\bar{K})\setminus S$ there exists $i\in\{1,\ldots,n+1\}$ such that $\varphi_{i}(T)\cap\{\infty\}=\emptyset$.
\end{itemize}
\end{thm}

We conclude with a remark about potential improvements of the number $n+1$ of maps in Theorem~\ref{manymaps}.

\begin{remark}
Let $C$ be a curve over $\QB$, and let $S$ be a finite subset of $C(\QB)$.
If $S$ has the form $\varphi^{-1}(\br(\varphi))$ for some Belyi map $\varphi\colon C\to\Line$,
then we may replace the $n+1$ maps $\varphi_1,\ldots,\varphi_{n+1}$ in the conclusion of Theorem~\ref{manymaps} with the single map $\varphi$.
If $S$ does not have this form then the $n+1$ Belyi maps in the conclusion of Theorem~\ref{manymaps}
cannot be replaced by a smaller set of maps.
 For, if $\varphi_1,\ldots,\varphi_n\colon C\rightarrow\mathbb{P}^1$ are Belyi maps with $\br(\varphi_i)=\{0,1,\infty\}$ and
$S\subsetneq\varphi_i^{-1}(\{0,1,\infty\})$, then pick some
$t_i\in\varphi_{i}^{-1}(\{0,1,\infty\})\setminus S$, and note that
\[
T_0:=\{t_1,\ldots,t_n\}
\]
satisfies $\abs{T_0}\le n$ and
\[
\varphi_i(T_0)\cap\{0,1,\infty\}\neq\emptyset
\]
for every $i$.
\end{remark}


\end{document}